\documentclass[12pt]{amsart}

\usepackage[utf8]{inputenc}

\usepackage{amsthm, enumitem, fullpage, amsrefs, tabularx, mathrsfs, mathtools}
\usepackage{thmtools}

\usepackage[colorlinks=true,allcolors=blue]{hyperref}

\usepackage[capitalize]{cleveref}

\newtheorem{theorem}{Theorem}[section]
\newtheorem{lemma}[theorem]{Lemma}
\newtheorem{proposition}[theorem]{Proposition}
\newtheorem{corollary}[theorem]{Corollary}

\newtheorem*{maintheorem}{Main Theorem}

\theoremstyle{definition}
\newtheorem{definition}[theorem]{Definition}

\newtheorem{setup}[theorem]{Setup}

\theoremstyle{remark}
\newtheorem{remark}[theorem]{Remark}

\newcommand{\R}{\mathbb{R}}
\newcommand{\C}{\mathbb{C}}
\newcommand{\N}{\mathbb{N}}
\newcommand{\cone}{{\mathscr C}}
\newcommand{\bulk}{\mathscr K}
\newcommand{\cc}{{\rm c}}
\newcommand{\volform}[1]{d\mu^{#1}}
\newcommand{\spinor}[1]{\Sigma #1}
\newcommand{\norm}[1]{\left\| #1 \right\|}
\newcommand{\abs}[1]{\left| #1 \right|}
\newcommand{\innerprod}[2]{\left\langle #1, #2 \right\rangle}
\newcommand{\Ltwoinnerprod}[2]{\left( #1, #2 \right)_{L^2}}
\newcommand{\sphere}{{\mathbb{S}}}
\newcommand{\FredsaOp}{\mathcal F_{\rm sa}}
\newcommand{\double}{{\mathsf D}}
\newcommand{\boundedOp}[1]{\mathcal L\left( #1 \right)}

\newcommand{\cl}[1]{\mkern1.7mu\overline{\mkern-1.7mu #1 \mkern-0.2mu}\mkern0.2mu}
\newcommand{\metriccl}[1]{\mkern2.5mu\overline{\mkern-2.5mu #1 \mkern-0.1mu}\mkern0.1mu}
\newcommand{\dom}[1]{\mathrm{dom}\left(#1\right)}

\DeclareMathOperator{\Cl}{Cl}
\DeclareMathOperator{\Dirac}{\mathcal D}
\DeclareMathOperator{\Brue}{\mathcal B}
\DeclareMathOperator{\EE}{\mathbf{E}}
\DeclareMathOperator{\FF}{\mathbf{F}}
\DeclareMathOperator{\LL}{\mathbf{L}}
\DeclareMathOperator{\HH}{\mathbf{H}}
\DeclareMathOperator{\id}{id}
\DeclareMathOperator{\specflow}{sf}
\DeclareMathOperator{\tr}{tr}
\DeclareMathOperator{\spec}{spec}
\DeclareMathOperator{\scal}{scal}
\DeclareMathOperator{\curvOp}{\mathcal R}
\DeclareMathOperator{\dist}{dist}
\DeclareMathOperator{\PP}{\mathcal{P}}
\DeclareMathOperator{\PPP}{\mathscr{P}}
\DeclareMathOperator{\gap}{gap}
\DeclareMathOperator{\End}{End}
\DeclareMathOperator{\Lip}{Lip}
\DeclareMathOperator{\Isom}{Isom}
\DeclareMathOperator{\supp}{supp}

\DeclareMathOperator{\coneLinkOp}{\mathscr S}

\begin{document}

\title{Lipschitz rigidity for scalar curvature on singular manifolds in odd dimensions}

\author[Lukas Schönlinner]{Lukas Schönlinner}
\address[Lukas Schönlinner]{Universität Augsburg, Universitätsstr.~14, 86159 Augsburg, Germany}
\email{lukas.schoenlinner@uni-a.de}

\keywords{Abstract cone operator, cone-like singularities, Lipschitz maps, scalar curvature bound, spectral flow, twisted Dirac operator}
\subjclass{Primary: 53C23, 53C24; Secondary: 53C27, 58J30}

\begin{abstract}
	The main result of this article is a Llarull-type rigidity statement for scalar curvature on Riemannian spin manifolds with cone-like singularities in odd dimensions. 
	The even dimensional analog was proven in an earlier work together with Simone Cecchini, Bernhard Hanke and Thomas Schick using index theory and the analysis of abstract cone operators, which applies to Dirac operators associated with generalized cone metrics.

	We will extend the analysis of abstract cone operators, apply it to twisted Dirac operators on singular manifolds and combine it with a spectral flow argument to prove the main result.
\end{abstract}

\maketitle

\tableofcontents

\section{Introduction and main result}

In recent years, the study of rigidity phenomena for scalar curvature has attracted considerable attention.
The interplay between the Dirac operator and positive scalar curvature via index theory and the Schr\"odinger-Lichnerowicz formula has been especially fruitful in this context. 
A classic example for this is the theorem of Llarull, which has been prototypical for many subsequent results.

\begin{theorem}[\cite{Llarull}*{Theorem B}]\label{LlarullThm}
	Let $n\geq 2$ and let $(M,g)$ be a closed connected Riemannian spin manifold of dimension $n$ with $\scal_g \geq n(n-1)$.
	Furthermore, let $f\colon (M,g)\to (\sphere^n, g_{\sphere^n})$ be a smooth 1-Lipschitz map of non-zero degree to the round sphere.

	Then $f$ is a Riemannian isometry. 
\end{theorem}

This result was later generalized in several directions. 
Goette and Semmelmann \cite{GoetteSemmelmann02} proved that the target manifold can be replaced by certain Riemannian manifolds with positive curvature operators and nonnegative Euler characteristics. 

More recently, Larull’s theorem was generalized to manifolds with boundary, as in the work of Lott \cite{Lott}, who proved a Goette-Semmelmann-type in that case.
Another instance is a rigidity statement for warped product metrics, which was proven by Cecchini and Zeidler \cite{CecchiniZeidler24} in odd dimensions and by B\"ar, Brendle, Hanke and Wang \cite{BaerBrendleHankeWang24} for all dimensions. 

A further recent development is the extension of Llarull's theorem to metrics with lower regularity in the domain and to comparison maps that are merely Lipschitz.
See the work of Cecchini, Hanke and Schick \cite{CHS}, and a follow-up work together with the author \cite{CHSS} as well as the work of Lee and Tam \cite{LeeTam22}.
Also in line of these results is the work of B\"ar \cite{Baer_2024}, where comparison maps of Lipschitz regularity are considered.

These results were largely influenced by Gromov's work, and we refer to his \emph{Four Lectures} \cite{GromovFourLectures} for further reading.

In the present article, the aim is to prove a Llarull-type rigidity statement in which the domain is allowed to be non-compact with cone-like singularities.
This is a continuation of \cite{CHSS}*{Theorem 1.8}, where the even-dimensional analog of the Main Theorem below was shown. 
We now close the gap by proving the odd-dimensional case. 

Before we state the main result, we recall what we mean by Riemannian metrics with cone-like singularities. 

\begin{definition}
	Let $M$ be a closed smooth manifold. Let $\vartheta>0$ and let
	\begin{equation*}
		g_r\in C^\infty((0,\vartheta), C^\infty(M, T^*M\otimes T^*M))
	\end{equation*}
	be a smooth family of Riemannian metrics on $M$.
	Assume that there exists a $C^2$-Riemannian metric
	\begin{equation*}
		g_0 \in C^2(M, T^*M\otimes T^*M)
	\end{equation*}
	such that
	\begin{equation*}
		g_r \overset{r\to 0}{\longrightarrow} g_0 \qquad \text{in } C^2(M, T^*M\otimes T^*M)
	\end{equation*}
	and that for all $p\in M$ we have
	\begin{equation*}
		\lim_{r\to 0} \abs{r\cdot \partial_r g_r}_{(T_pM, g_r(p))} = 0 \qquad\text{and}\qquad \lim_{r\to 0} \abs{r^2\cdot \partial_r^2 g_r}_{(T_pM, g_r(p))} = 0.
	\end{equation*}
	The metric $\cone g_r := dr^2 + r^2g_r$ on $M\times (0,\vartheta)$ is called a \emph{generalized cone metric}.

	If the family $g_r = g$ is constant in $r$, we will call the metric $\cone g $ a \emph{straight cone metric}. 
\end{definition}

\begin{definition}\label{ConeManifold}
	A smooth Riemannian manifold $(N, G)$ is a \emph{compact Riemannian manifold with cone-like singularities} if
	\begin{itemize}
		\item there is a compact smooth submanifold $\bulk\subset N$ with boundary and if
		\item for some $0<\vartheta \leq 1$, there is a Riemannian isometry
		      \begin{equation*}
			      \nu \colon (N\setminus \bulk, G\vert_{N\setminus \bulk}) \to (\partial \bulk\times (0,\vartheta), \cone g_r),
		      \end{equation*}
		      where $\cone g_r$ is a generalized cone metric.
	\end{itemize}
	We refer to $\partial\bulk$ as the \emph{link} and to $\bulk$ as the \emph{bulk}.
\end{definition}

Let $(N, G)$ be a connected smooth Riemannian manifold with cone-like singularities. 
We denote its metric completion with respect to the path metric induced by $G$ with $\metriccl{N}$.
Furthermore, we refer to the finite set $\metriccl{N}\setminus N = \{x_1, \ldots, x_\ell\}$ as the \emph{cone points} of $N$. 
If $f\colon (N,G)\to (X,g_X)$ is a Lipschitz continuous map to a smooth closed Riemannian manifold $X$, then $\metriccl{f}\colon \metriccl{N}\to X$ denotes its unique continuous continuation. 
Here the Lipschitz constant is understood with respect to the path metrics on $N$ and $X$.

We define the degree of $f$ in the following way.
Let $0<\varepsilon<\frac\pi2$ such that, for $\varepsilon$-balls around distinct cone points, we have $B_\varepsilon\left(\metriccl{f}(x_i)\right)\cap B_\varepsilon\left(\metriccl{f}(x_j)\right) = \emptyset$ for all $i\neq j$. 
By possibly increasing $\bulk$, we can achieve $f(N\setminus\mathrm{int}(\bulk))\subset V:= \bigcup_{i=1}^\ell B_\varepsilon\left(\metriccl{f}(x_i)\right)$.
If $N$ is oriented, then $\deg(f)$ is defined as the homological mapping degree of the map of pairs $f\colon (\bulk, \partial\bulk)\to (\sphere^{n+1}, V)$. 

By Rademacher's theorem, $f$ is differentiable almost everywhere.
If $f$ is differentiable at $p\in N$, we say that $f$ is \emph{area non-increasing} at $p$ if the induced map on two-vectors
\begin{equation*}
	\Lambda^2 d_pf \colon \Lambda^2T_pN \to \Lambda^2T_{f(p)}X
\end{equation*}
has norm smaller or equal to 1. 	
It was already realized by Llarull \cite{Llarull}*{Theorem C} that, in dimension $n\geq 3$, \cref{LlarullThm} holds true under the weaker assumption that $f$ is area non-increasing instead of 1-Lipschitz.

\begin{maintheorem}\label{MainTheorem}
	Let $n\geq 2$ be even and let $(N,G)$ be a compact connected oriented Riemannian manifold with cone-like singularities of dimension $n+1$ admitting a spin structure with $\scal_G \geq (n+1)n$. 
	Furthermore, let $f\colon (N,G)\to (\sphere^{n+1}, g_{\sphere^{n+1}})$ be a Lipschitz continuous map of non-zero degree to the round sphere that is area non-increasing almost everywhere. 
	
	Let $\metriccl{f}\colon \metriccl{N}\to \sphere^{n+1}$ be the metric extension of $f$ and let $\metriccl{N}\setminus N= \{x_1, \ldots, x_\ell\}$ be the cone points of $N$.
	
	Then $f$ is a smooth Riemannian isometry onto $\sphere^{n+1}\setminus \left\{\metriccl f(x_1), \ldots \metriccl f(x_\ell)\right\}$. In particular, $N$ is diffeomorphic to a punctured sphere.
\end{maintheorem}

There are some similarities to a version of Llarull's theorem for $L^\infty$-metrics in \cite{ChuLeeZhu24}. 
However, the setting is quite different. 
The metric $G$ in \cref{ConeManifold} is an $L^\infty$-metric only if the link $\partial \bulk$ is empty or consists of a disjoint union of spheres. 
On the other hand, $L^\infty$-metrics can be more flexible in a neighborhood of the singular points than generalized cone metrics. 

A different Llarull type comparison result, which applies to manifolds with iterated conical singularities and non-empty boundaries, is contained in \cite{JovanovicWang}.

The article is structured as follows.
In \cref{sec:AbstractCone}, we recall the definition of abstract cone operators from \cite{CHSS} and extend the analysis with conditions that guarantee self-adjointness of the operator and compactness of the resolvent.

Subsequently, it is shown in \cref{sec:DiracOnSingularManifolgs} that the Dirac operator on an odd-dimensional manifold with cone-like singularities, twisted by certain pullback bundles, is a self-adjoint abstract cone operator with discrete spectrum and finite multiplicities.
In particular, this implies the Fredholmness of the operator. 

\cref{sec:DeformationOfSingularDirac} is concerned with families of twisted Dirac operators that are continuous in the gap topology. 
This ensures a well-defined spectral flow.
Building on this, we perform a deformation of an arbitrary cone metric to a straight cone metric and a deformation of the Lipschitz comparison map to a smooth map that is constant in a neighborhood of the singular points. 
This plays a key role in the argument, because, together with a formula of Getzler \cite{Getzler}, it enables us in \cref{sec:SpectralFlowFormula} to do a doubling argument and reduce the computation of the spectral flow to the closed case treated in \cite{LiSuWang} and \cite{Baer_2024}.

Finally, everything is put together in \cref{sec:MainTheoremProof} to prove the Main Theorem.

\subsubsection*{Acknowledgements}

This is part of my doctoral dissertation project and I like to thank my advisor Bernhard Hanke for his continuous support and Helge Frerichs for his feedback on an earlier version of this article.
I am also grateful to the University of G\"ottingen for their hospitality during the summer term 2025, and to Thomas Schick for very helpful discussions. 

This work was supported by the SPP 2026 "Geometry at Infinity" funded by the Deutsche Forschungsgemeinschaft (DFG, German Research Foundation).


\section{Abstract cone operators}\label{sec:AbstractCone}

We recall the definition of abstract cone operators from \cite{CHSS}.
The construction of these singular operators is based on ideas from \cites{BrueningSeeley87,Bruening90}.

Let $\EE, \FF$ be separable Hilbert spaces with decompositions
	\begin{equation*}
		\EE = \EE_{\rm bulk} \oplus \EE_{\rm cone}, \qquad \FF = \FF_{\rm bulk} \oplus \FF_{\rm cone}
	\end{equation*}
	and assume that there exists a separable Hilbert space $\LL$ together with a real number $0<\vartheta\leq1$ and two Hilbert space isometries
	\begin{equation*}
		\Phi_{\EE} \colon \EE_{\rm cone} \overset{\sim}{\longrightarrow} L^2((0,\vartheta), \LL), \quad \Phi_{\FF} \colon \FF_{\rm cone} \overset{\sim}{\longrightarrow} L^2((0,\vartheta), \LL).
	\end{equation*}
	We identify $\EE_{\rm cone}$ and $\FF_{\rm cone}$ with $L^2((0,\vartheta), \LL)$ via these isometries.
	
	For $\varphi \in C_\cc^\infty([0,\vartheta), \R)$ and $u = (u_b,u_c)\in \EE_{\rm bulk}\oplus L^2((0,\vartheta), \LL)$, we define
	\begin{align*}
		\varphi \cdot u =&\, (0, \varphi \cdot u_c) \in \EE,\\
		(1-\varphi)\cdot u =&\, (u_b, (1-\varphi) \cdot u_c)\in \EE.
	\end{align*}
	
	Furthermore, let 
	\begin{equation*}
		S_0 \colon \LL\supset \dom{S_0}\to \LL
	\end{equation*}
	be an essentially self-adjoint operator and let
	\begin{equation*}
		S_1\colon (0,\vartheta) \to \boundedOp{\dom{S_0},\LL}
	\end{equation*}
	be an essentially bounded measurable map. 
	Here $\dom{S_0}$ is equipped with the graph norm of $S_0$.
	Let $\cl S_0 \colon \LL\supset \dom{\cl S_0}\to \LL$ be the unique self-adjoint extension of $S_0$ and, for $r\in (0,\vartheta)$, let $\cl S_1(r)\colon \LL\to \LL$ be the unique bounded extension of $S_1(r)$.

	We define the operators
	\begin{equation*}
		\coneLinkOp_0 \colon L^2((0,\vartheta), \LL) \supset C_\cc^\infty((0,\vartheta), \dom{S_0}) \to L^2((0,\vartheta), \LL), \quad \coneLinkOp_0 u(r) :=S_0 u(r)
	\end{equation*}
	and
	\begin{equation*}
		\coneLinkOp_1 \colon L^2((0,\vartheta), \LL) \supset C_\cc^\infty((0,\vartheta), \dom{S_0}) \to L^2((0,\vartheta), \LL), \quad \coneLinkOp_1 u(r) :=S_1(r) u(r), \quad \mathrm{a.e.}
	\end{equation*}

	By $C_{\cc,0}^\infty([0,\vartheta),\R)$, we denote the space of compactly supported smooth functions on $[0,\vartheta)$ that are equal to 1 in a neighborhood of 0.

\begin{definition}
	A closeable densely defined linear operator
	\begin{equation*}
		\Brue\colon \EE \supset \dom{\Brue} \to \FF
	\end{equation*}
	is called an \emph{abstract cone operator} with link operator $S_0$ and perturbation $S_1$ if the following conditions hold.
	\begin{enumerate}[start=0,label=\textup{(AC\arabic*)}]
		\item (Locality of domain) For all $\varphi\in C_{\cc,0}^\infty([0,\vartheta),\R)$, we have $\varphi\cdot \dom{\Brue}\subset \dom{\Brue}$. \label{ACOpLocalityOfDomain}
		\item (Domain over cone part) We have
		\begin{equation*}
			C_\cc^\infty((0,\vartheta), \dom{S_0})\subset \dom{\Brue},
		\end{equation*} 
		and the inclusion 
		\begin{equation*}
			\varphi \cdot C_\cc^\infty((0, \vartheta), \dom{S_0}) \subset \varphi\cdot \dom{\Brue}
		\end{equation*}
		is dense with respect to the graph norm of $\Brue$ for all $\varphi\in C_{\cc,0}^\infty([0,\vartheta),\R)$.
		\item (Locality of operator) We have 
		\begin{equation*}
			\Brue \left(C_\cc^\infty((0,\vartheta), \dom{S_0})\right)\subset \FF_{\rm cone},
		\end{equation*}
		and for all $\varphi, \psi \in C_{\cc,0}^\infty([0,\vartheta),\R)$ with $\psi\vert_{\supp \varphi} = 1$, it follows that
		\begin{equation*}
			\varphi \Brue(1-\psi) = 0.
		\end{equation*}
		\item (Product structure over the cone part) The equality 
		\begin{equation*}
			\Brue u(r) = \partial_r u(r) + \tfrac1r \left(S_0 + S_1(r)\right)u(r)
		\end{equation*}
		holds for all $u\in C_\cc^\infty((0,\vartheta), \dom{S_0})$ and almost all $r\in (0,\vartheta)$.
		\item (Spectral gap) The spectrum of $\cl S_0$ satisfies
		\begin{equation*}
			\spec\left(\cl S_0\right) \cap \left[-\tfrac12, \tfrac12\right] = \emptyset.
		\end{equation*}\label{ACOpSpectralGap}
		\item (Small perturbation) For almost all $r\in (0,\vartheta)$, the operator 
		\begin{equation*}
			\cl S_0^{-1} S_1(r) \colon \dom{S_0} \to \LL
		\end{equation*}
		extends to a bounded operator $\cl S_0^{-1}\cl S_1(r)\colon \LL\to \LL$, and we have the operator norm bounds
		\begin{align*}
			\norm{\cl S_0^{-1}\cl S_1}_{L^\infty((0,\vartheta), \boundedOp{\LL})}\leq& \inf_{s\in \spec(\cl S_0)} \abs{\frac{2s+1}{4s}}, \\
			\norm{\cl S_1 \cl S_0^{-1}}_{L^\infty((0,\vartheta), \boundedOp{\LL})}\leq& \inf_{s\in \spec(\cl S_0)} \abs{\frac{2s-1}{4s}}.
		\end{align*}\label{ACOpSmallPerturbation}
	\end{enumerate}
\end{definition}

There are closable operators
\begin{align*}
	\partial_r\colon L^2((0,\vartheta), \LL) &\supset C_\cc^\infty((0, \vartheta), \LL) \to L^2((0,\vartheta), \LL),\\
	\tfrac1r \coneLinkOp_0\colon L^2((0,\vartheta), \LL) &\supset C_\cc^\infty((0, \vartheta), \dom{S_0}) \to L^2((0,\vartheta), \LL)
\end{align*}
with minimal closed extensions $\cl\partial_r$ and $\cl{\frac1r \coneLinkOp_0}$.
Furthermore, the restriction of $\Brue$ to smooth sections with compact support on the cone part defines a densely defined closable operator
\begin{equation*}
	\Brue_{\rm cone} \colon L^2((0,\vartheta), \LL) \supset C_\cc^\infty((0, \vartheta), \dom{S_0}) \to L^2((0,\vartheta), \LL).
\end{equation*}

We recall the main results about abstract cone operators from \cite{CHSS}*{Theorem 2.7}.

\begin{theorem}\label{AbstractConePrevResults}
	If $\Brue$ is an abstract cone operator, then
	\begin{enumerate}[label = \textup{(\alph*)}]
		\item $\dom{\cl \Brue_{\rm cone}} = \dom{\cl\partial_r}\cap \dom{\cl{\frac1r \coneLinkOp_0}}$.\label{AbstractConeDomainNearCone}
		\item The graph norm on $ \dom{\cl \Brue_{\rm cone}}$ is equivalent to the Sobolev $1$-norm
		\begin{equation*}
     		\norm{u}_{H^1_{\rm cone}} := \left(  \norm{u}^2_{L^2((0,\vartheta), \LL)} + \norm{\cl\partial_r u}^2_{L^2((0,\vartheta), \LL)}  + \norm{\cl{\tfrac{1}{r} \coneLinkOp_0} u}^2_{L^2((0,\vartheta), \LL)}\right)^{\tfrac{1}{2}} .
		\end{equation*} 
		\item For every $\varphi\in C_{\cc,0}^\infty([0,\vartheta),\R)$, the graph norm on $\dom{\cl\Brue}$ is equivalent to the Sobolev $1$-norm
		\begin{equation*}
			\norm{u}_{H^1} := \left( \norm{u}^2_{\EE} + \norm{\cl\Brue((1-\varphi)u)}_{\FF}^2 + \norm{\cl\partial_r (\varphi\cdot u)}^2_{L^2((0,\vartheta), \LL)} + \norm{\cl{\tfrac1r \coneLinkOp_0} (\varphi\cdot u)}^2_{L^2((0,\vartheta), \LL)}\right)^{\tfrac12}.
		\end{equation*}
		\item If $\cl\Brue$ admits an interior parametrix (see \cite{CHSS}*{Definition 2.4}), then $\cl \Brue$ is Fredholm.
	\end{enumerate}
\end{theorem}

In this section we extend the cone analysis with the following results.

\begin{theorem}\label{AbstractConeResults}
	Let $\Brue$ be an abstract cone operator. 
	\begin{enumerate}[label = \textup{(\alph*)}]
		\item If $\Brue$ is a symmetric operator such that, for some $\varphi\in C_{\cc,0}^\infty([0,\vartheta),\R)$, the domain of the adjoint satisfies $(1-\varphi)\cdot \dom{\Brue^*} \subset \dom{\cl\Brue}$, then $\cl\Brue$ is self-adjoint. \label{AbstractConeOpSelfAdjoint}
		\item If the inclusion $\dom{\cl S_0}\hookrightarrow \LL$ is compact and if there is some $\varphi\in C_{\cc,0}^\infty([0,\vartheta), \R)$ such that the inclusion $(1-\varphi)\cdot \dom{\cl\Brue}\hookrightarrow \EE$ is compact, then the inclusion $\dom{\cl\Brue}\hookrightarrow \EE$ is compact. \label{AbstracConeOpCompactIncl}
	\end{enumerate}
\end{theorem}

\begin{remark}
	The spectral gap condition in \ref{ACOpSpectralGap} is sometimes referred to as the \emph{geometric Witt condition} and plays a key role in the proof of \cref{AbstractConeResults} \ref{AbstractConeOpSelfAdjoint}. 
	Without this condition, the analysis of cone operators gets much more delicate (compare \cites{Bruening90, BrueningSeeley87}). 
\end{remark}


\subsection{Self-adjoint cone operators}

Let us recall some notation and basic results from \cite{CHSS}*{Section 2}.
Due to \cite{CHSS}*{Lemma 2.10}, we can assume without loss of generality that $S_0$ is self-adjoint.
For $s>-\frac12$, we define the integral operator $\PP_{0,s}\in \boundedOp{L^2(0,1),\C}$ by
\begin{equation*}
	\PP_{0,s}u(r) := \int_0^r \left(\tfrac yr\right)^s u(y)\,dy.
\end{equation*}
Similarly, for $s<\frac12$, let $\PP_{1,s}\in \boundedOp{L^2(0,1),\C}$ be the integral operator 
\begin{equation*}
	\PP_{1,s}u(r) := \int_1^r \left(\tfrac yr\right)^s u(y)\,dy.
\end{equation*}
Using the direct integral decomposition, we write
\begin{equation*}
	L^2((0,1),\LL) = \int_{\spec(S_0)} L^2((0,1),\LL_s)\,d\mu(s),
\end{equation*}
and for $u(s)\in L^2((0,1),\LL_s)$, we define
\begin{equation*}
	\PPP u(s) := \begin{cases}
		\PP_{0,s}u(s), & \text{if } s >\tfrac12, \\
		\PP_{1,s}u(s), & \text{if } s<-\tfrac12.
	\end{cases}
\end{equation*}
This yields a measurable family $\PPP(u)\in (L^2((0,1),\LL_s))_{s\in \spec(S_0)}$, which defines a bounded operator $\PPP\in \boundedOp{L^2((0,1),\LL)}$.
Moreover, the compositions $\frac1r \coneLinkOp_1\PPP$ and $\PPP \tfrac1r \coneLinkOp_1$ are well-defined and extend to bounded operators on $L^2((0,1),\LL)$.
We denote these extensions by $\left\{ \tfrac1r \coneLinkOp_1\PPP \right\}$ and $\left\{ \PPP \tfrac1r \coneLinkOp_1 \right\}$, respectively.
They satisfy the operator norm bounds
\begin{equation}\label{eq:OpNormsS1P}
	\norm{\left\{ \tfrac1r \coneLinkOp_1\PPP \right\}}_{\boundedOp{L^2((0,1),\LL)}} \leq \tfrac12\qquad\text{and} \qquad \norm{\left\{ \PPP \tfrac1r \coneLinkOp_1 \right\}}_{\boundedOp{L^2((0,1),\LL)}} \leq \tfrac12.
\end{equation}

The operator $\PPP$ is an almost right parametrix for $\Brue$ near the cone tip in the sense that, for every $\psi\in C_\cc^\infty([0,\vartheta),\R)$ and every $u\in L^2((0,\vartheta),\LL)$, we have
\begin{equation}\label{eq:rightConeParametrix}
	\cl \Brue\left(\psi\cdot \PPP\right) u = \psi\cdot u + \psi \cdot \left\{ \tfrac1r \coneLinkOp_1\PPP \right\} u + \psi'\cdot \PPP u.
\end{equation}

\begin{lemma}\label{LeftConeAdjointParametrix}
	There exists a bounded operator $V\in \boundedOp{L^2((0,\vartheta),\LL)}$ such that, for every $v\in \dom{\Brue^*}$ and every $\psi\in C_\cc^\infty([0,\vartheta),\R)$, we have
	\begin{equation*}
		\psi \cdot v = \PPP^*V(\psi\Brue^* - \psi')v.
	\end{equation*}
\end{lemma}

\begin{proof}

	Let $u\in L^2((0,1),\LL)$ and set
	\begin{align*}
		A:= \left\{ \PPP\tfrac1r \coneLinkOp_1 \right\}, \\
		B:= \left\{ \tfrac1r \coneLinkOp_1\PPP \right\}.
	\end{align*}
	Iterating \eqref{eq:rightConeParametrix}, we obtain
	\begin{equation*}
		\Ltwoinnerprod{\psi u}{v} = \sum_{j=0}^k\Ltwoinnerprod{(-1)^j A^j \PPP u}{(\psi \Brue^* -\psi' )v} + \Ltwoinnerprod{(-1)^{k+1}B^{k+1}u}{v}.
	\end{equation*}

	Since for the $L^2$-norms we have $\norm{A} \leq \frac12$ and $\norm{B}\leq \frac12$ by \eqref{eq:OpNormsS1P}, the series $\sum_{j=0}^k (-1)^j A^j$ converges in operator norm to a bounded operator $\widetilde V\in \boundedOp{L^2((0,1),\LL)}$ and the second term converges in operator norm to $0$ as $k\to \infty$.
	Consequently, this yields
	\begin{equation*}
		\Ltwoinnerprod{ u}{\psi \cdot v} = \Ltwoinnerprod{u}{\PPP^*V(\psi B^* - \psi')v}
	\end{equation*}
	with $V = \widetilde V^*$.
	This concludes the proof since $u$ was arbitrary.
\end{proof}

Observe that $\PP_{0,s}^* = -\PP_{1,-s}$ for $s>-\frac12$ and $\PP_{1,s}^* = -\PP_{0,-s}$ for $s<\frac12$ as operators on $L^2((0,1),\C)$.
Thus, for $u(s) \in L^2((0,1),\LL) = (L^2((0,1),\LL_s))_{s\in \spec(S_0)}$, it follows
\begin{equation*}
	\PPP^* u(s) = \begin{cases}
		-\PP_{1,-s}u(s), & \text{if } s >\tfrac12, \\
		-\PP_{0,-s}u(s), & \text{if } s<-\tfrac12.
	\end{cases}
\end{equation*}

For every $u\in C_\cc^0((0,1),\C)\otimes \dom{S_0}$, we have $\PPP^* u \in C^1((0,1),\dom{S_0})$ and 
	\begin{equation*}
		(\partial_r -\tfrac1r \coneLinkOp_0)\PPP^* u = -u,
	\end{equation*}
compare \cite{CHSS}*{Proposition 2.18}.

Furthermore, the operators
	\begin{equation*}
		\tfrac1r \coneLinkOp_0\PPP^* \colon C_\cc^0((0,1),\C)\otimes \dom{S_0} \to L^2((0,1),\LL)
	\end{equation*}
	and
	\begin{equation*}
		\tfrac1r \coneLinkOp_1 \PPP^* \colon C_\cc^0((0,1),\C)\otimes \dom{S_0} \to L^2((0,1),\LL)
	\end{equation*}
	extend to bounded operators $\left\{ \frac1r\coneLinkOp_0\PPP^* \right\}$ and $\left\{ \frac1r\coneLinkOp_1\PPP^* \right\}$ on $L^2((0,1),\LL)$.
	Their operator norms satisfy
	\begin{equation*}
		\norm{\left\{ \tfrac1r \coneLinkOp_0 \PPP^* \right\}}_{\boundedOp{L^2((0,1),\LL)}} \leq \sup_{s\in \spec(S_0)} \frac{\abs{s}}{\abs{s-\tfrac12}}
	\end{equation*}
	and
	\begin{equation*}
		\norm{\left\{ \tfrac1r \coneLinkOp_1 \PPP^* \right\}}_{\boundedOp{L^2((0,1),\LL)}} \leq \norm{S_0^{-1}S_1}_{L^\infty((0,\vartheta), \boundedOp{\LL})}\cdot \sup_{s\in \spec(S_0)} \frac{\abs{s}}{\abs{s-\tfrac12}}\ .
	\end{equation*}

	This leads to the following analog of \eqref{eq:rightConeParametrix} for $\PPP^*$. Its proof is similar to \cite{CHSS}*{Proposition 2.26}.

\begin{proposition}\label{RightConeAdjointParametrix}
	For every $\psi \in C_\cc^\infty([0,\vartheta), \R)$ and every $u\in L^2((0,\vartheta),\LL)$, we have $\psi\cdot \PPP^*u\in \dom{\cl\Brue}$ and
	\begin{equation*}
		\cl\Brue \left(\psi\cdot  \PPP^*\right) u = -\psi\cdot u + 2\psi\cdot \left\{ \tfrac1r \coneLinkOp_0\PPP^* \right\} + \psi \cdot \left\{ \tfrac1r \coneLinkOp_1\PPP^* \right\} + \psi'\cdot \PPP^* u.
	\end{equation*}

	In particular, the map $u\mapsto \psi\cdot \PPP^* u$ defines a bounded operator $L^2((0,\vartheta),\LL)\to \dom{\cl\Brue}$.
\end{proposition}

\begin{proof}[Proof of \cref{AbstractConeResults} \ref{AbstractConeOpSelfAdjoint}]
	The inclusion $\dom{\cl\Brue}\subset \dom{\Brue^*}$ is automatic.
	Thus, we need to show the reverse inclusion $\dom{\Brue^*}\subset \dom{\cl\Brue}$.

	Let $v\in \dom{\Brue^*}$.
	By assumption, $(1-\varphi)\cdot v \in \dom{\cl\Brue}$ for some $\varphi \in C_{\cc,0}^\infty([0,\vartheta),\R)$.
	It remains to prove $\varphi v \in \dom{\cl\Brue}$.
	Choose a cutoff function $\psi\in C_\cc^\infty([0,\vartheta), \R)$ such that $\psi \varphi = \varphi$.
	\cref{LeftConeAdjointParametrix} yields
	\begin{equation*}
		\varphi\cdot v = \varphi\cdot \PPP^*V(\psi\Brue^* - \psi')v
	\end{equation*}
	with $V \in \boundedOp{L^2((0,\vartheta),\LL)}$.
	\cref{RightConeAdjointParametrix} yields $\varphi\cdot  v \in \dom{\cl\Brue}$.
\end{proof}


\subsection{Compact inclusion of the domain}

Let $\LL$ be a separable Hilbert space and let $S \colon \LL\subset \dom{S} \to \LL$ be a self-adjoint operator. 
Furthermore, for $a<b$, let $H_0^1((a,b),\LL)$ be the domain of the minimal closed extension $\cl \partial_r$ of the operator
\begin{equation*}
	\partial_r \colon L^2((a,b),\LL)\supset C_\cc^\infty((a,b),\LL) \to L^2((a,b),\LL).
\end{equation*}

We define the operator  
\begin{equation*}
	\mathscr S \colon L^2((a,b),\LL)\supset C_\cc^\infty((a,b),\dom{S}) \to L^2((a,b),\LL), \quad \mathscr S u(r) := S u(r)
\end{equation*}
whose minimal closed extension satisfies
\begin{equation*}
	\dom{\cl{\mathscr S}} = L^2((a,b),\dom{S}).
\end{equation*}

Consider the Hilbert space
\begin{equation*}
	\HH:= H^1_0((a,b),\LL)\cap L^2((a,b),\dom{S})
\end{equation*}
equipped with the scalar product
\begin{equation*}\label{AbstractConeOpSobolevScalarProd}
	\innerprod{u}{v}_{\HH}:= \Ltwoinnerprod{\cl\partial_r u}{\cl\partial_r v} + \Ltwoinnerprod{\cl{\mathscr S} u}{\cl{\mathscr S} v} + \Ltwoinnerprod{u}{v}.
\end{equation*}

\begin{lemma}\label{CompactInclusionCylinderOp}
	Assume that the inclusion $\iota\colon \dom{S}\to \LL$ is compact. 
	Then the inclusion
	\begin{equation*}
		\tau\colon \HH \hookrightarrow L^2((a,b),\LL)
	\end{equation*}
	is also compact.  
\end{lemma}

\begin{proof}
	We show that $\tau$ can be approximated in $\boundedOp{\HH,L^2((a,b),\LL)}$ by finite rank operators. 

	Let 
	\begin{equation*}
		\kappa \colon H^1_0((a,b),\LL)\to L^2((a,b),\LL)
	\end{equation*}
	be the inclusion. 
	The map $\iota$ induces an inclusion $L^2((a,b),\dom{S})\hookrightarrow L^2((a,b),\LL)$ which we continue to denote by $\iota$.
	We have $\tau = \kappa \circ \iota\vert_{\HH}$.

	Since $\iota$ is a compact operator between the Hilbert spaces $\dom{S}$ and $\LL$, we can approximate $\iota$ in $\boundedOp{\dom{S}, \LL}$ by a sequence of finite rank operators $(\iota_n)_{n\in \N}$. 
	Therefore, $\LL_n :=\iota_n(\dom{S})\subset \LL$ is a finite dimensional subspace and, for every $n$, we obtain induced operators
	\begin{equation*}
		\iota_n \colon L^2((a,b),\dom{S})\to L^2((a,b), \LL_n).
	\end{equation*}
	Observe that $\iota_n(\HH) \subset H_0^1((a,b), \LL_n)$ and $\kappa (\iota_n(\HH))\subset L^2((a,b), \LL_n)$.
	Due to the Rellich-Kondrachov theorem, the inclusion
	\begin{equation*}
		\kappa\vert_{H_0^1((a,b), \LL_n)}\colon H_0^1((a,b), \LL_n) \to L^2((a,b),\LL_n) 
	\end{equation*}
	is compact.
	Hence for every $n\in \N$, we can approximate $\kappa \circ \iota_n$ in $\boundedOp{\HH, L^2((a,b),\LL)}$ by a sequence $(\kappa_m^n)_{m\in \N}$ of finite rank operators. 
	We will define a sequence $(\kappa_n)_{n\in \N}$ inductively:
	Let $\kappa_1:= \kappa_1^1$.
	Suppose we have defined $\kappa_{n-1}$ for $n\in \N$. 
	Choose $m_n\in \N$ such that 
	\begin{equation*}
		\norm{\kappa\circ\iota_n - \kappa_{m_n}^n}_{\boundedOp{\HH, L^2((a,b),\LL)}}\leq \tfrac12\norm{\kappa\circ \iota_{n-1} - \kappa_{n-1}}_{\boundedOp{\HH, L^2((a,b),\LL)}}.
	\end{equation*} 
	Set $\kappa_n := \kappa_{m_n}^n$.
	Consequently, we get
	\begin{equation*}
		\norm{\kappa\circ\iota_n - \kappa_n}_{\boundedOp{\HH, L^2((a,b),\LL)}}\to 0 \qquad \text{as }n\to \infty.
	\end{equation*}

	For brevity, we will suppress the interval $(a,b)$ from the notation in the following computation.
	Let $\varepsilon>0$ and $u\in \HH$. 
	Choose $n$ large enough such that $\norm{\iota - \iota_n}_{\boundedOp{\dom{S}, \LL}} \leq \frac\varepsilon 2$ and $\norm{\kappa\circ \iota_n - \kappa_n}_{\boundedOp{H_0^1(\LL), L^2( \LL)}} \leq \frac\varepsilon 2$.
	Then it holds
	\begin{align*}
		\norm{(\kappa\circ \iota - \kappa_n)u}_{L^2(\LL)} \leq&\, \norm{\kappa(\iota - \iota_n)u}_{L^2(\LL)} + \norm{(\kappa \circ \iota_n - \kappa_n)u}_{L^2(\LL)} \\
		\leq& \, \norm{\iota - \iota_n}_{\boundedOp{\dom{S}, \LL}}\norm{u}_{L^2(\dom{S})} + \norm{\kappa\circ \iota_n - \kappa_n}_{\boundedOp{H_0^1(\LL), L^2( \LL)}}\norm{u}_{H_0^1(\LL)} \\
		\leq& \, \varepsilon \norm{u}_{\HH}.
	\end{align*}
	We deduce that $\norm{\tau - \kappa_n}_{\boundedOp{\HH, L^2((a,b), \LL)}}\to 0$ as $n\to \infty$, which concludes the proof.
\end{proof}

Now let $a = 0$, $b = \vartheta$ and consider the operator
\begin{equation*}
	\tfrac1r \mathscr S \colon L^2((0,\vartheta), \LL)\supset C_\cc^\infty((0,\vartheta), \dom{S}) \to L^2((0,\vartheta), \LL).
\end{equation*}
Observe that the domain of its minimal closed extension is given by the weighted Sobolev space
\begin{equation*}
	\dom{\cl{\tfrac1r \mathscr S}} = r\cdot L^2((0,\vartheta), \dom{S}).
\end{equation*}

Let
\begin{equation*}
	\widetilde \HH := H_0^1((0, \vartheta),\LL)\cap \dom{\cl{\tfrac1r \mathscr S}}
\end{equation*}
equipped with the scalar product
\begin{equation*}
	\innerprod{u}{v}_{\widetilde \HH}:= \Ltwoinnerprod{\cl\partial_r u}{\cl\partial_r v} + \Ltwoinnerprod{\cl{\tfrac1r \mathscr S} u}{\cl{\tfrac1r \mathscr S} v} + \Ltwoinnerprod{u}{v}.
\end{equation*}

\begin{corollary}\label{CompactInclusionConeOp}
	If the inclusion $\dom{S}\hookrightarrow \LL$ is compact, then the inclusion
	\begin{equation*}
		\widetilde \HH \hookrightarrow L^2((0,\vartheta), \LL)
	\end{equation*}
	is compact.
\end{corollary}

\begin{proof}
	The inclusion $\widetilde \HH \hookrightarrow L^2((0,\vartheta), \LL)$ factors through the inclusion $\HH \hookrightarrow L^2((0,\vartheta), \LL)$.
	Furthermore, the inclusion $\dom{\cl{\frac1r \mathscr S}} \hookrightarrow \dom{\cl {\mathscr S}}$ is bounded since $\frac1r \geq \frac1\vartheta > 0$.
	Thus, the claim follows from \cref{CompactInclusionCylinderOp}.
\end{proof}

\begin{proof}[Proof of \cref{AbstractConeResults} \ref{AbstracConeOpCompactIncl}]
	Let $(u_n)_{n\in \N}$ be a bounded sequence in $\dom{\cl \Brue}$.
	By assumption, we find a subsequence $(u_{n_k})_{k\in \N}$ such that the sequence $((1-\psi)u_{n_k})_{k\in \N}$ converges in $\EE$.
	Due to \cref{AbstractConePrevResults} \ref{AbstractConeDomainNearCone}, we have $\psi u_{n_k}\in \widetilde\HH$ for each $k$.
	By \cref{CompactInclusionConeOp}, we find a subsequence $(u_{n_{k_\ell}})_{\ell\in \N}$ such that $(\psi u_{n_{k_\ell}})_{\ell\in \N}$ converges in $\EE$.
	Consequently, the subsequence $(u_{n_{k_\ell}})_{\ell\in \N}$ converges in $\EE$.
\end{proof}


\section{Dirac operator on manifolds with cone-like singularities}\label{sec:DiracOnSingularManifolgs}

The aim of this section is to prove that the Dirac operator on manifolds with cone-like singularities and twisted with certain pullback bundles is a self-adjoint abstract cone operator with discrete spectrum. 

We fix the following setup.
\begin{setup}\label{GeomSetup}
	Let
	\begin{itemize}
		\item $(N^{n+1},G)$ be a smooth compact spin manifold with cone-like singularities, where $n=2k\geq 2$ is even,
		\item $\Gamma$ be a smooth Riemannian metric on $N$ such that $\Gamma\vert_\cone = dr^2 + r^2 \gamma$, where $\gamma$ is a smooth Riemannian metric on $M$,
		\item $(W,h)$ be a smooth closed Riemannian manifold of dimension $n+1$,
		\item $f\colon (N, G)\to (W, h)$ be a $\Lambda$-Lipschitz map for some $\Lambda >0$,
		\item $\left(E_0, \nabla^{E_0}\right)\to W$ be a smooth hermitian vector bundle with metric connection.
	\end{itemize}
	We use the following notation
	\begin{itemize}
		\item $\Sigma N\to N$ denotes the spinor bundle associated with $\Gamma$ and $\spinor{M}\to M$ denotes the spinor bundle associated with $\gamma$.
		\item Let $\nabla^{\spinor N,G}$ be the spinor connection on $\spinor N$ associated with $G$ and let $\nabla^{\spinor M, \gamma}$ be the spinor connection on $\spinor M$ associated with $\gamma$ (compare \cite{CHS}*{Definition 4.3}).
		\item The pullback bundle $\left(E,\nabla^E\right) := f^*\left(E_0, \nabla^{E_0}\right)\to M$ is a hermitian Lipschitz vector bundle with metric Lipschitz connection.
		\item $L_G^2(\spinor N\otimes E)$ denotes the $L^2$-space of square integrable sections of $\spinor N\otimes E$ with respect to the metric $G$.
	\end{itemize}
\end{setup}

As $G$ is a smooth Riemannian metric, we could define the Dirac operator directly on the spinor bundle corresponding to $G$.
However, since we are working in the setting of abstract cone operators, it is more convenient to define the Dirac operator associated with $G$ on the spinor bundle associated with $\Gamma$.
Both descriptions are equivalent and can be transformed into each other.

We recall the definition of Dirac operators twisted with Lipschitz bundles and Lipschitz connections, which were first introduced in \cite{CHS}.

There is a unique positive $\Gamma$-self-adjoint automorphism $b_G\colon TN \to TN$ such that 
\begin{equation*}
	G_p(b_G X, b_G Y) = \Gamma_p(X,Y)\qquad \text{for all }p\in N \text{ and }X,Y\in T_pN.
\end{equation*}
We will use the notation
\begin{equation}\label{NotationMetricMediator}
	X^G := b_G(X).
\end{equation}
Similarly, $X^{g_r}$ is defined for all $X\in TM$.
Moreover, we set
\begin{equation*}
	\nabla^{\spinor N\otimes E}:= \nabla^{\spinor N, G}\otimes \id_E + \id_{\spinor N}\otimes \nabla^E.
\end{equation*}

The Dirac operator associated with $G$ and twisted with $(E, \nabla^E)$
\begin{equation}\label{DefinitionTwistedDirac}
	\Dirac_E \colon L^2_G(\spinor N \otimes E) \supset \Lip_\cc(\spinor N \otimes E)\to L^2_G(\spinor N \otimes E)
\end{equation}
is locally defined by the formula
\begin{equation*}
	\Dirac_E = \sum_{i=1}^{n+1} e_i \cdot \nabla^{\spinor N\otimes E}_{e_i^G}\ .
\end{equation*}
Here $(e_1\ldots, e_{n+1})$ is a local $\Gamma$-orthonormal frame of $TN$.

The following proposition is the main result of this section.

\begin{proposition}\label{ConeDiracSAFredholmOp}
	Assume $\scal_{g_0} > 1$.	
	Then the minimal closed extension
	\begin{equation*}
		\cl\Dirac_E \colon L^2_G(\spinor N\otimes E)\supset \dom{\cl \Dirac_E}\to L^2_G(\spinor N\otimes E)
	\end{equation*}
	is self-adjoint with spectrum consisting of eigenvalues with finite multiplicities.
	
	In particular, $\cl\Dirac_E$ is Fredholm.
\end{proposition}

\begin{proof}
	We first prove that $\Dirac_E$ is an abstract cone operator.
	The proof is analogous to \cite{CHSS}*{Section 4.2}.
	For the sake of completeness, we recall the main steps of the proof.

	We use the notation from \cref{ConeManifold}.
	Let $\cone := M\times (0,\vartheta)$ and set
	\begin{itemize}
		\item $\EE = \FF := L^2_G(\spinor N\otimes E)$,
		\item $\EE_{\rm bulk} = \FF_{\rm bulk} := L^2_G(\bulk, \spinor N\otimes E)$,
		\item $\EE_{\rm cone} = \FF_{\rm cone} := L^2_G(\cone, \spinor N\otimes E)$.
	\end{itemize}

	As preparation, we first consider the untwisted Dirac operator
	\begin{equation*}
		\Dirac \colon C_\cc^\infty(\spinor{N}) \to C_\cc^\infty( \spinor{N})
	\end{equation*}
	associated with the metric $G$.
	Furthermore, let
	\begin{equation*}
		D_\gamma \colon L^2_\gamma(\spinor{M}) \supset  C^\infty(\spinor M) \to L^2_\gamma(\spinor{M})
	\end{equation*}
	be the Dirac operator corresponding to $\gamma$. 
	Similarly, for each $r\in (0,\vartheta)$, let
	\begin{equation*}
		D_{g_r} \colon L^2_{g_r}(\spinor{M}) \supset  C^\infty(\spinor M) \to L^2_{g_r}(\spinor{M})
	\end{equation*}
	be the Dirac operator associated with $g_r$.

	Let
	\begin{equation*}
		\rho^G := \frac{\volform{G}}{\volform{\Gamma}}
	\end{equation*}
	be the volume density with respect to the metric $\Gamma$.
	On $\cone = M\times (0,\vartheta)$, we have
	\begin{equation*}
		\rho_r := r^{-n}\rho^G\vert_\cone = \frac{\volform{g_r}}{\volform{\gamma}}.
	\end{equation*}

	We consider $\spinor N\vert_M \cong \spinor M$ as a $\Cl(M,\gamma)$-module.
	Using parallel transport along geodesic lines, we obtain an isometry
	\begin{equation*}
		\Psi \colon L^2_G(\cone, \spinor N) \overset{\sim}{\longrightarrow} L^2((0,\vartheta), L^2_{r^2g_r}(\spinor M))
	\end{equation*}
	such that
	\begin{equation*}
		\Psi \circ \Dirac\circ\, \Psi^{-1} = \partial_r + \tfrac1r D_{g_r} - \tfrac n2 H_r\, ,
	\end{equation*}
	where $H_r = \frac1n \tr\left(-\nabla\partial_r\right)$ denotes the mean curvature (compare \cite{BaerGauduchonMoroianu}). 

	Post-composing the isometry $\Psi$ by multiplication with $\left(r^n\rho_r\right)^{\frac 12}$, one obtains an isometry
	\begin{equation}\label{eq:IsometryUntwisted}
		\widetilde \Psi \colon L^2_G(\cone, \spinor N) \overset{\sim}{\longrightarrow} L^2((0,\vartheta), L^2_{\gamma}(\spinor M))
	\end{equation}
	such that
	\begin{equation*}
		\widetilde\Psi \circ \Dirac \circ\, \widetilde\Psi^{-1} = \partial_r + \tfrac1r \left(\rho_r\right)^{\frac12}\circ D_{g_r} \circ \left(\rho_r\right)^{-\frac12}.
	\end{equation*}
	Put 
	\begin{equation*}
		D_r := \left(\rho_r\right)^{\frac12}\circ D_{g_r}\circ \left(\rho_r\right)^{-\frac12}\colon L^2_\gamma(\spinor{M}) \supset  C^\infty(\spinor M) \to L^2_\gamma(\spinor{M}).
	\end{equation*}
	Its minimal closed extension $\cl D_r$ is self-adjoint with $\dom{\cl D_r} = H^1_\gamma(\spinor M)$.

	Now we consider the twisted Dirac operator $\Dirac_E$.
	Due to standard considerations, the operator $\Dirac_E$ is symmetric and in particular closable.

	Recall that $\{x_1, \ldots, x_\ell\} = \metriccl N\setminus N$ denotes the set of cone points of the metric completion $\metriccl N$, and $\metriccl{f}$ refers to the unique continuous continuation of $f$ to $\metriccl N$.

	Let $0<\varepsilon < {\rm inj}(W)$, where ${\rm inj}(W)$ denotes the injectivity radius of $W$, and let $V_i := B_\varepsilon\left(\metriccl f(x_i)\right)\subset W$ be open balls centered at $\metriccl f(x_i)$.
	For $\varepsilon$ small enough, we achieve $V_i\cap V_j = \emptyset$ whenever $\metriccl f(x_i) \neq \metriccl f(x_j)$.
	Put
	\begin{equation*}
		V:= \bigcup_{i = 1}^\ell V_i\subset W.
	\end{equation*}

	For every $i = 1,\ldots \ell$, we isometrically identify the fibre  
	\begin{equation*}
		F_i:= E_0\vert_{\cl f(x_i)}
	\end{equation*}
	with a fixed fibre $F$ and obtain a unitary trivialization
	\begin{equation}\label{eq:trivializationE0}
		E_0\vert_V \cong V\times F
	\end{equation}
	using parallel translation along geodesic lines emanating from $\cl f(x_i)\in W$.
	This yields a unitary trivialization
	\begin{equation}\label{eq:trivializationE}
		E\vert_{f^{-1}(V)}\cong f^{-1}(V)\times F
	\end{equation}
	and, after possibly shrinking $\cone$, we can assume that $\cone \subset f^{-1}(V)$.

	For $v\in T_xM$, let 
	\begin{equation*}
		\bar v (r) = (v, 0)\in T_{(x,r)}\cone \cong T_xM\oplus \R.
	\end{equation*}

	Let $\omega\in \Omega^1(V, \End(F))$ be the smooth connection $1$-form of $\nabla^{E_0}$ with respect to \eqref{eq:trivializationE0}.
	Furthermore, let $(e_1\ldots ,e_n)$ be a local $\gamma$-orthonormal frame and set
	\begin{equation*}
		\LL := L^2_\gamma(\spinor M\otimes F).
	\end{equation*}

	The isometry \eqref{eq:IsometryUntwisted} together with the trivialization \eqref{eq:trivializationE} induces an isometry
	\begin{equation}\label{ConeOpIsometry}
		\Phi \colon L^2_G(\cone, \spinor N\otimes E) \overset{\sim}{\longrightarrow} L^2((0,\vartheta), \LL)
	\end{equation}
	such that on $C_\cc^\infty((0, \vartheta), C^\infty(\spinor M\otimes F))$
	\begin{equation*}
		\Phi\circ \Dirac_E \circ\, \Phi^{-1} = \partial_r + \tfrac1r  D_r\otimes \id_F + \sum_{i=1}^n c(e_i) \otimes \left(f^*\omega\right)\left(\tfrac1r \bar e_i^{g_r}\right) + c(\partial_r)\otimes \,\left(f^*\omega\right)(\partial_r)
	\end{equation*}
	for all $(x,r)\in \cone$ where $f$ is differentiable.
	Here $c(-)$ denotes Clifford multiplication.
	Set
	\begin{equation*}
		S_0 :=  D_0\otimes \id_F
	\end{equation*}
	with $\dom{S_0} = C^\infty(\spinor M\otimes F)$ and 
	\begin{equation*}
		S_1(r):= ( D_r -  D_0)\otimes \id_F + \, r\left(\sum_{i=1}^n c(e_i) \otimes \left(f^*\omega\right)\left(\tfrac1r \bar e_i^{g_r}\right) + c(\partial_r)\otimes \,\left(f^*\omega\right)(\partial_r)\right). 
	\end{equation*}

	The spectral gap condition \ref{ACOpSpectralGap} is satisfied due to the fact that $\scal_{g_0}>1$.
	The map $S_1\colon (0,\vartheta)\to \boundedOp{\dom{S_0}, \LL}$ is continuous, and its operator norm is bounded since $\omega$ is bounded and $f$ is $\Lambda$-Lipschitz.
	Consequently, we get
	\begin{align*}
		\abs{S_1(r)\cl S_0^{-1}}_{L^\infty((0,\vartheta), \boundedOp{\LL})} &\to 0, \\
		\abs{\cl S_0^{-1}S_1(r)}_{L^\infty((0,\vartheta), \boundedOp{\LL})} &\to 0
	\end{align*}
	as $r\to 0$ since $r\mapsto \cl D_r\in \boundedOp{H_\gamma^1(\spinor M), L_\gamma^2(\spinor M)}$ is continuous.

	It is now straightforward to check the axioms \ref{ACOpLocalityOfDomain}-\ref{ACOpSmallPerturbation}, see \cite{CHSS}*{Section 4.2}.

	Next, we verify that $\cl\Dirac_E$ is self-adjoint.
	As mentioned before, $\Dirac_E$ is symmetric.
	Let $\varphi \in C_{\cc}^\infty([0,\vartheta),\R)$ be a cutoff function with $\varphi \equiv 1$ near $r=0$.
	Then we find a compact submanifold $\Omega\subset N$ such that $\supp(1-\varphi)\subset \Omega$ and 
	\begin{equation*}
		(1-\varphi)\cdot \dom{\Dirac_E^*}\subset (1-\varphi)\cdot H^1(\Omega, \spinor N\otimes E)\subset \dom{\cl\Dirac_E}.
	\end{equation*}
	Together with \cref{AbstractConeResults} \ref{AbstractConeOpSelfAdjoint}, we deduce that $\cl\Dirac_E$ is self-adjoint.

	Finally, we show that the spectrum of $\cl\Dirac_E$ consists of eigenvalues with finite multiplicities.
	The inclusion $\dom{\cl D_0} = H^1_\gamma(\spinor M \otimes F) \subset L^2_\gamma(\spinor M\otimes F)$ is compact and
	\begin{equation*}
		(1-\varphi)\cdot \dom{\cl\Dirac_E}\subset  H^1(\Omega, \spinor N\otimes E)\subset L^2(\Omega, \spinor N\otimes E) \subset L^2_G(\spinor N\otimes E)
	\end{equation*}
	is compact by the Rellich-Kondrachov theorem. 	
	It follows that the inclusion $\dom{\cl\Dirac_E}\subset L^2_G(\spinor N\otimes E)$ is compact by \cref{AbstractConeResults} \ref{AbstracConeOpCompactIncl}.
	Since $\cl\Dirac_E$ is self-adjoint, the claim follows now from a standard argument. 
\end{proof}


\section{Continuous families of Dirac operators on manifolds with cone-like singularities}

\subsection{Deformation of Dirac operators}\label{sec:DeformationOfSingularDirac}

The spectral flow can be defined for one-parameter families of self-adjoint Fredholm operators that are continuous with respect to the gap \mbox{topology}. For the definition of the spectral flow, we refer to \cite{BoosLeschPhillips}.

First, we recall the definition of the gap topology (compare \cite{BoosLeschPhillips}).
Let $H$ be a separable Hilbert space and let us denote the space of self-adjoint unbounded Fredholm operators on $H$ by $\FredsaOp(H)$.
We equip $\FredsaOp(H)$ with the gap topology, which is induced by the metric 
\begin{equation*}
	d_{\gap}(D_1, D_2):= \norm{(D_1 + i)^{-1} - (D_2 + i)^{-1}}_{\boundedOp{H}}.
\end{equation*}

\begin{lemma}\label{FamilyUnboundedOpGapTop}
	For $m\geq0$, let $0\in U\subset \R^m$ be a compact connected subset. 
	Let
	\begin{equation*}
		 U \to \FredsaOp(H), \quad s\mapsto D_s
	\end{equation*}
	be a map and set $D := D_0$.
	Assume that the domain $\dom{D_s} = \dom{D}$ is independent of $s$, the graph norm of $D_s$ is independent of $s$ up to equivalence and 
	\begin{equation*}
		U\ni s\mapsto D_s \in \boundedOp{\dom{D}, H}
	\end{equation*}
	is continous with respect to the graph norm of $D$.

	Then $s\mapsto D_s$ is continuous in the gap topology.
\end{lemma}

\begin{proof}
	Since $D_s$ is self-adjoint for every $s$, the operator
	\begin{equation*}
		(D_s + i)^{-1} \colon H\to \dom{D}
	\end{equation*}
	is a bounded operator with respect to the graph norm on $\dom{D}$. 
	By compactness of $U$, there is a constant $C>0$ such that for every $s\in U$
	\begin{equation*}
		\norm{(D_s +i)^{-1}}_{\boundedOp{H}}\leq \norm{(D_s+i)^{-1}}_{\boundedOp{H,\dom{D}}}\leq C.
	\end{equation*}

	We compute
	\begin{align*}
		d_{\gap}(D_s,D_{s'})  \leq & \, \norm{(D_s + i)^{-1}}_{\boundedOp{H}}\norm{D_{s'}- D_s}_{\boundedOp{\dom{D},H}} \norm{(D_{s'} + i)^{-1}}_{\boundedOp{H, \dom{D}}} \\
		\leq                       & \, C^2\norm{D_{s'}- D_s}_{\boundedOp{\dom{D},H}}.
	\end{align*}
	Since $s\mapsto D_s$ is continuous in the graph norm of $D$, it follows that it is also continuous with respect to the gap topology.
\end{proof}

We want to perform several deformations of the twisted Dirac operator $\Dirac_E$ defined in \eqref{DefinitionTwistedDirac}. 
For $m\geq 0$, let $0\in U\subset \R^m$ be a compact connected subset.
\begin{enumerate}[label = (\alph*)]
	\item (Deformation of the twist connection) Let
\begin{equation*}
	U \to \boundedOp{H^1(E_0), L^2(T^*W\otimes E_0)}, \quad s\mapsto \nabla^{s}
\end{equation*}
be a continuous family of smooth metric connections on $E_0$.
	\item (Deformation of the comparison map) Let 
\begin{equation*}
	U\to \Lip(N,W),\quad s\mapsto f_s
\end{equation*}
be a continuous family of Lipschitz functions such that $f_0 = f$ and $\dist(f(x), f_s(x))<{\rm inj}(W)$ for every $(s,x)\in U\times N$.
Here the space of Lipschitz functions is equipped with the $C^{0,1}$-topology.
For $(s,x)\in U\times N$, consider the parallel transport $(E_0)_{f(x)}\to (E_0)_{f_s(x)}$ along the unique shortest geodesic connecting $f(x)$ with $f_s(x)$ with respect to $\nabla^0$.
This induces a continuous family of Lipschitz vector bundle isomorphisms
\begin{equation*}
	P_s\colon U \to \Isom\left(f^*E_0, (f_s)^*E_0\right)
\end{equation*}
covering the identity on $N$.
We define a family of Lipschitz connections $\nabla^{E,s}:= (P_s)^{-1} \circ (f_s)^*\nabla^s \circ\, P_s$ on the hermitian Lipschitz bundle $E:= f^* E_0$.
	\item (Deformation of the cone metric) Let $G_s$, $s\in U$, be a continuous family of smooth Riemannian metrics with cone-like singularities on $N$ and let $g_r^s$ be a family of Riemannian metrics on $M$ such that $G_s = dr^2 + r^2g_r^s$ near the cone tips.
\end{enumerate}

For every $s\in U$, we consider the Dirac operator
\begin{equation}\label{eq:ConeDiracDeformation}
	\Dirac_{E,s}\colon \Lip_\cc(\spinor N\otimes E)\to L^2_{G_s}(\spinor N\otimes E)
\end{equation}
associated with $G_s$ and twisted with $\left(E, \nabla^{E,s}\right)$.
\cref{ConeDiracSAFredholmOp} shows that for every $s\in U$, the minimal closed extension $\cl \Dirac_{E,s}$ is a self-adjoint Fredholm operator on $L^2_{G_s}(\spinor N\otimes E)$.
Multiplication with the square root of the volume density $\volform{G}/\volform{G_s}$ induces an isometry
\begin{equation}\label{FamilyOfLTwoIsometries}
	\sigma^s\colon L^2_{G_s}(\spinor N\otimes E) \overset{\sim}{\longrightarrow} L^2_G(\spinor N\otimes E).
\end{equation}

\begin{proposition}\label{ConeDiracGapTopology}
	The map
	\begin{equation*}
		U \to  \FredsaOp\left(L^2_G(\spinor N\otimes E)\right), \quad s \mapsto \sigma^s \circ\cl\Dirac_{E,s}\circ \left(\sigma^s\right)^{-1}
	\end{equation*}
	is continuous with respect to the gap topology.
\end{proposition}

\begin{proof}
	Since $P_s$ is induced by parallel transport with respect to $\nabla^0$, we have
	\begin{equation*}
		\nabla^{E,s} := (f_s)^*\nabla^0 + (P_s)^{-1} \circ (f_s)^*\left(\nabla^s - \nabla^0\right) \circ\, P_s.
	\end{equation*}

	For every $s\in U$, let $\eta_s\in \Omega^1(V,\End(E_0))$ be the local 1-form of $\nabla^s - \nabla^0\in \Omega^1(V,\End(E_0))$ and let $\omega\in \Omega^1(V, \End(E_0))$ be the connection $1$-form of $\nabla^0$ with respect to \eqref{eq:trivializationE0}.
	It follows that the assignment  
	\begin{equation*}
		s\mapsto \omega_s:= (f_s)^*\omega + (P_s)^{-1}\left(f_s^*\eta_s\right) P_s\in L^\infty(V, T^*N\otimes \End(E))
	\end{equation*}
	is continuous. 

	Near the cone tips, conjugation with the isometry \eqref{ConeOpIsometry} yields
	\begin{equation*}
		\Phi_s\circ \Dirac_{E,s}\circ\, \Phi_s^{-1} = \partial_r + \tfrac1r \left(S_{0,s} + S_{1,s}(r)\right).
	\end{equation*}
	Here
	\begin{equation*}
		S_{0,s} = D_{0,s} \otimes \id_F,
	\end{equation*}
	where $D_{0,s}$ is the Dirac operator on $M$ associated with the metric $g_0^s$, and
	\begin{equation*}
		S_{1,s}(r) = (D_{r,s} - D_{0,s})\otimes \id_F +\, r\left(\sum_{i=1}^n c(e_i) \otimes \omega_s\left(\tfrac1r \bar e_i^{g_r^s}\right) + c(\partial_r)\otimes \,\omega_s(\partial_r)\right)\quad {\rm a.e.}
	\end{equation*}
	Observe that $\dom{S_{0,s}}$ is independent of $s$ and the graph norm of $S_{0,s}$ is independent of $s$ up to equivalence. 
	Put $\mathfrak{D}_{\rm link} = \dom{S_{0,0}} = H_{g_0^0}^1(\spinor M\otimes F)$ equipped with the graph norm of $S_{0,0}$.
	The maps
	\begin{equation*}
		U\to \boundedOp{\mathfrak{D}_{\rm link}, \LL}, \quad s\mapsto S_{0,s}
	\end{equation*}
	and
	\begin{equation*}
		U\to  L^\infty((0,\vartheta), \boundedOp{\mathfrak{D}_{\rm link}, \LL}), \quad s\mapsto S_{1,s}
	\end{equation*}
	are continuous. 
	Furthermore, the domain $\dom{\Dirac_{E,s}}$ is independent of $s$ and, for every $\varphi\in C_\cc^\infty([0,\vartheta),\R)$ with $\varphi \equiv 1$ near $0$, the assignment
	\begin{equation*}
		U\to \boundedOp{(1-\varphi)\cdot \dom{\Dirac_{E,0}}, L^2_G(\spinor N\otimes E)}, \quad s\mapsto \sigma^s \circ\Dirac_{E,s}\circ \left(\sigma^s\right)^{-1}
	\end{equation*}
	is continuous since the coefficients of $\Dirac_{E,s}$ depend continuously on $s$.

	We apply \cite{CHSS}*{Proposition 2.38}, which carries over to families that are indexed by a compact subset $U$. 
	Hence $\dom{\sigma^s \circ\cl\Dirac_{E,s}\circ \left(\sigma^s\right)^{-1}}$ is independent of $s$, the graph norm of $\sigma^s \circ\bar\Dirac_{E,s}\circ \left(\sigma^s\right)^{-1}$ on $\dom{\cl\Dirac_{E,0}} = \dom{\sigma^s \circ\cl\Dirac_{E,s}\circ \left(\sigma^s\right)^{-1}}$ is independent of $s$ up to equivalence and the map
	\begin{equation*}
		U\to \boundedOp{\dom{\cl \Dirac_{E,0}}, L^2_G(\spinor N\otimes E)}, \quad s\mapsto \sigma^s \circ\cl\Dirac_{E,s}\circ \left(\sigma^s\right)^{-1}
	\end{equation*}
	is continuous.

	Applying \Cref{FamilyUnboundedOpGapTop} concludes the proof.
\end{proof}

\subsection{Formula for the spectral flow}\label{sec:SpectralFlowFormula}

We continue to work in the \cref{GeomSetup} with $W = \sphere^{n+1}$ equipped with the round metric $h = g_{\sphere^{n+1}}$.

Assume that $\scal_G \geq C$ for some constant $C>0$.
Moreover, let $E_0= \spinor \sphere^{n+1}$ be the canonical spinor bundle with spin connection $\nabla^{E_0} = \nabla^{\spinor \sphere^{n+1}}$.
The following construction goes back to \cite{Baer_2024}.
For $s\in \R$, we consider the family of metric connections
\begin{equation*}
	\nabla^s_X := \nabla_X^{E_0} + \left(s-\tfrac12\right)\cdot c(X),
\end{equation*}
where $X\in T\sphere^{n+1}$ and $c(-)$ denotes Clifford multiplication.
Set
\begin{equation}\label{Eq::FamilyOfConnections}
	(E,\nabla^{E,s}) := (f^*E_0, f^*\nabla^s) \to N.
\end{equation}

Recall that $n = 2k$.
For $s=0$ and $s= 1$, there are $2^k$ parallel spinors for $\nabla^0$ and $\nabla^1$, respectively, called \emph{Killing spinors} and hence trivialize $E_0$.
This induces Lipschitz trivializations 
\begin{equation*}
	U_+ \colon E \to \underline{\C}^{2^k}\qquad\text{and} \qquad U_- \colon E\to \underline{\C}^{2^k}.
\end{equation*}
We obtain unitary Lipschitz vector bundle isomorphisms
\begin{align*}
	V_+:=& \id_{\spinor N}\otimes\, U_+ \colon \spinor N \otimes E \to \spinor N\otimes \underline{\C}^{2^k},\\
	V_-:=& \id_{\spinor N}\otimes\, U_- \colon \spinor N \otimes E \to \spinor N\otimes \underline{\C}^{2^k}.
\end{align*}

Let
\begin{equation}\label{FamilyOfDiracOpWithPathOfConnOnSphere}
	\Dirac_{E,s} \colon L^2_G(\spinor N\otimes E)\supset\Lip_\cc(\spinor N\otimes E) \to L^2_{G}(\spinor N\otimes E)
\end{equation}
be the Dirac operator on $N$ twisted with $\left(E, \nabla^{E,s}\right)$.
For brevity, we will omit the isometry $\sigma^s$ defined in \eqref{FamilyOfLTwoIsometries} from the notation. 
Due to \cref{ConeDiracSAFredholmOp}, the minimal closed extension $\cl\Dirac_{E,s}$ is a self-adjoint Fredholm operator for every $s\in [0,1]$.
Furthermore, according to \cref{ConeDiracGapTopology}, the path
\begin{equation*}
	[0,1]\ni s\mapsto \cl\Dirac_{E,s}\in \FredsaOp\left(L^2_G(\spinor N\otimes E)\right)
\end{equation*}
is continuous in the gap topology. 
Therefore, its spectral flow $\specflow\left([0,1]\ni s\mapsto \cl\Dirac_{E,s}\right)$ is well-defined.
Note that this path of operators is explicitly given by
\begin{equation}\label{eq:DiracPathExplicit}
	\Dirac_{E,s} = \Dirac_{E,0} + \left(s -\tfrac12\right)\sum_{i=1}^{n+1}c(e_i)\otimes c\left(f_* e_i^G\right)
\end{equation}
for $(e_1, \ldots, e_{n+1})$ a local $\Gamma$-orthonormal frame of $TN$ (compare \cite{Baer_2024}*{Section 4}). 
Set
\begin{equation*}
	B_s := \left(s-\tfrac12\right) \sum_{i=1}^{n+1}c(e_i)\otimes c\left(f_* e_i^G\right).
\end{equation*}

We apply \cite{Baer_2024}*{Lemma 1}, which is a local computation and thus carries over to our non-compact setting, to obtain 
\begin{equation}\label{eq:DiracSpectrumEqualities}
	\spec\left(\cl\Dirac_{E,0}\right) = \spec\left(\cl\Dirac_{E,1}\right) = \spec\left(\cl \Dirac\right).
\end{equation}
Here $\Dirac$ denotes the untwisted Dirac operator. 
Furthermore, for $V := V_-^{-1}V_+$, we have
\begin{equation*}
	V^{-1} \circ \Dirac_{E,0}\circ\, V = \Dirac_{E,1}.
\end{equation*}
The vector bundle isomorphism $V$ induces a unitary operator on $L^2_G(\spinor N\otimes E)$, that we continue to denote with $V$, such that
\begin{equation}\label{eq:DiracUnitaryConjugation}
	\cl \Dirac_{E,s} = (1-s) \cl \Dirac_{E,0} + s V^* \cl \Dirac_{E,0} V.
\end{equation}

\begin{lemma}\label{ConeDiracSpectralGap}
	The untwisted Dirac operator satisfies the spectral gap
	\begin{equation*}
		\spec\left(\cl\Dirac\right)\cap \left[-\tfrac12 \sqrt{\tfrac{n+1}n C},\tfrac12 \sqrt{\tfrac{n+1}n C}\right] = \emptyset.
	\end{equation*}

\end{lemma}
\begin{proof}
	The proof is analogous to the case of closed manifolds.
	We use the integrated Schr\"odinger-Lichnerowicz formula for manifolds with cone-like singularities (see \cite{CHSS}*{Theorem 4.12}) and the fact that the spectrum of $\cl\Dirac$ is discrete due to \cref{ConeDiracSAFredholmOp}.
\end{proof}

We denote the spectral flow with $\specflow\left([0,1]\ni s\mapsto \cl\Dirac_{E,s}\right)$. For a precise definition, see \cite{BoosLeschPhillips}.

\begin{lemma}\label{specFlowDeformation}
	There exists a smooth Riemannian metric $\widetilde G$ with straight-cone singularities on $N$ and a smooth map $\widetilde f\colon N\to \sphere^{n+1}$ that is constant on an open subset $U\subset N$, such that its closure $\metriccl U$ in the metric completion $\metriccl N$ contains all the cone points, and such that the associated family of Dirac operator $\Dirac_{E,s,\widetilde f}$, constructed in \eqref{FamilyOfDiracOpWithPathOfConnOnSphere}, satisfies
	\begin{equation*}
		\specflow\left([0,1]\ni s\mapsto \cl\Dirac_{E,s}\right) = \specflow\left([0,1]\ni s\mapsto \cl\Dirac_{E,s,\widetilde f}\right).
	\end{equation*}
\end{lemma}

\begin{proof}
	In the following we use the notation from \cref{ConeManifold}.

	Let $\gamma$ be a smooth Riemannian metric on $M$ and let $\varphi\in C_\cc^\infty([0,\vartheta))$ be a cut-off function satisfying $\varphi \equiv 1$ on $\left[0,\frac\vartheta2\right]$.
	For $t\in [0,1]$, we define
	\begin{equation*}
		g_{r,t} := \varphi(r)\left(t\gamma + (1-t)g_r\right) + (1-\varphi(r))g_r.
	\end{equation*}

	Due to \cite{CHSS}*{Proposition 4.1} and $\scal_G \geq C >0$, we have $\scal_{g_0}\geq n(n-1)$.
	Thus, we can choose $\gamma$ in such a way that $\scal_{g_{0,t}} > 1$ for all $t\in [0,1]$.

	We set
	\begin{equation*}
		G_t\vert_\cone := \nu^*(dr^2 + r^2 g_{r,t})
	\end{equation*}
	and 
	\begin{equation*}
		G_t\vert_{N\setminus \cone} := G\vert_{N\setminus \cone}.
	\end{equation*}
	
	Obviously, we have $G_0 = G$ and 
	\begin{equation*}
		G_1\vert_{\cone'} = \nu^*(dr^2 + r^2 \gamma)
	\end{equation*}
	for $\cone' = M\times \left(0,\frac\vartheta2\right)$.

	Recall that $\{x_1,\ldots, x_\ell\} = \metriccl N\setminus N$ is the set of cone points of the metric completion $\metriccl N$ and $\metriccl f\colon \bar N\to \sphere^{n+1}$ denotes the unique continuous continuation of $f$. 
	For every $y\in \sphere^{n+1}\setminus \{\metriccl f(x_1), \ldots, \metriccl f(x_\ell)\}$, let $B_a(y)\subset B_b(y)\subset \sphere^{n+1}$ be open balls of radii $0<a<b<\pi$ which are centered at $y$ and satisfy $\metriccl f(x_i)\in B_a(y)$ for all $i=1,\ldots,\ell$.
	Choose a cutoff function $\chi\in C_\cc^\infty((0,b],[0,1])$ such that $\chi \equiv 1$ on $[a,b]$.
	For $t\in [0,1]$, we define a family of maps
	\begin{equation*}
		\widetilde\Phi_t \colon B_b(y)\to B_b(y), \qquad (r,x)\mapsto (\chi(tr + (1-t)a)r,x)
	\end{equation*}
	using the description $(r,x)\in B_b(y) \cong [0,b)\times \sphere^n/\{0\}\cup \sphere^n$.
	We obtain a smooth family of maps
	\begin{equation*}
		\Phi_t\colon \sphere^{n+1}\to \sphere^{n+1}
	\end{equation*}
	by extending $\widetilde\Phi_t$ by the identity. 
	It follows that $\deg(\Phi_t) = 1$ for all $t$ and $\Phi_0 = \id_{\sphere^{n+1}}$.

	Let $V = f^{-1}(B_a(y))$ and let $U \subset  N$ be an open set containing all cone points such that the closure satisfies $ \metriccl U \subset V$. 
	The set $N\setminus V$ is compact.
	Choose a continuous path of Lipschitz functions $f_t\colon N \to \sphere^{n+1}$ such that for every $(t,x)\in [0,1]\times N$, we have that
	\begin{itemize}
		\item $\dist(f(x), f_t(x))< \pi - b$,
		\item $f_0 = f$,
		\item $f_1$ is smooth on $N\setminus V$ and
		\item $f_t\vert_{U} = f\vert_U$.
	\end{itemize}

	Consider the family of Lipschitz maps
	\begin{equation*}
		\widetilde f_t := \Phi_t \circ f_t \colon N \to \sphere^{n+1}.
	\end{equation*}
	By construction, we have $\deg(\widetilde f_1) = \deg( f)$, $\widetilde f_0 = f$ and $\dist(f(x), \widetilde f_t(x)) < \pi$ for all $(t,x)\in [0,1]\times N$.
	Moreover, $\widetilde  f_1$ is smooth and constant on a neighborhood of the cone points and maps all cone points to the same image point. 

	We apply \cref{ConeDiracGapTopology} to obtain a homotopy of twisted Dirac operators $[0,1]^2\ni (s,t)\mapsto \Dirac_{E,s,t}$, which is continuous in the gap topology. 
	We have $\Dirac_{E,s,0} = \Dirac_{E,s}$ and $\Dirac_{E,s,1}$ is a Dirac operator associated with a straight-cone metric and twisted with $\big(\widetilde  f_1\big)^*(\spinor\sphere^{n+1}, \nabla^s)$.
	This concludes the proof by the homotopy invariance of the spectral flow (see \cite{BoosLeschPhillips}).
\end{proof}

\begin{proposition}\label{SpecFlowFormula}
	We have
	\begin{equation*}
		\abs{\specflow\left([0,1]\ni s\mapsto \cl\Dirac_{E,s}\right)} = \abs{\deg(f)}.
	\end{equation*}
\end{proposition}

	\begin{proof}
		Due to \cref{specFlowDeformation}, we can assume that $\Dirac_{E,s}$ is associated with a straight-cone metric and that $f$ is smooth and constant on an open subset $U\subset N$ that contains all the cone points.
		Set $y_0 = f(U)$.

		From \cite{Chou}, it follows that the heat kernel of $\cl\Dirac_{E,s}$ is trace-class.
		Equation \eqref{eq:DiracUnitaryConjugation} and \cite{Getzler}*{Corollary 2.7} imply that
		\begin{equation}\label{GetzlerFormula}
			\specflow\left([0,1]\ni s\mapsto \cl\Dirac_{E,s}\right) = \left(\tfrac\varepsilon \pi\right)^{1/2}\int_0^1  \tr\left(B_s e^{-\varepsilon\cl\Dirac_{E,s}^2}\right)\, ds
		\end{equation}
		for every $\varepsilon >0$.
		
		Let $\Omega\subset N$ be a compact submanifold of $N$ with boundary such that $N\setminus U\subset {\rm int}(\Omega)$.
		Let $\double \Omega$ be the double of $\Omega$ and define $\double F \colon \double \Omega \to \sphere^{n+1}$ by 
		\begin{equation*}
			\double F = \begin{cases}
				f \qquad &\text{on }\Omega,\\
				y_0 &\text{on } \double \Omega\setminus \Omega.
			\end{cases}
		\end{equation*}
		Observe that $\deg\left(\double F\right) = \deg(f)$.
		Furthermore, let $\double G$ be a smooth Riemannian metric on $\double \Omega$ that extends $G\vert_{\Omega}$.
		We denote by 
		\begin{equation*}
			[0,1]\ni s\mapsto \Dirac_{\double E,s}
		\end{equation*}
		the Dirac operator associated with $\double G$ and twisted with $\left(\double F\right)^*(\spinor \sphere^{n+1}, \nabla^s)$.
		By construction, $\Dirac_{\double E,s}$ and $\Dirac_{E,s}$ coincide when restricted to $\Omega$.
		Since $B_s$ vanishes on $U$, the restriction $B_s\vert_\Omega$ extends trivially to $\double \Omega$, and equation \eqref{GetzlerFormula} implies
		\begin{equation*}
			\specflow\left([0,1]\ni s\mapsto \cl\Dirac_{E,s}\right) = \specflow\left([0,1]\ni s\mapsto \cl\Dirac_{\double E,s}\right).
		\end{equation*}
		
		The statement now follows from \cite{Baer_2024}*{Section 5.2}.
	\end{proof}

\section{Proof of Main Theorem}\label{sec:MainTheoremProof}

Consider \cref{GeomSetup} with $W = \sphere^{n+1}$ and with $f\colon (N,G)\to (\sphere^{n+1}, g_{\sphere^{n+1}})$ a Lipschitz continuous map that is area non-increasing and has non-zero degree.
Moreover, let $\left(E_0, \nabla^{E_0}\right) = \left(\spinor \sphere^{n+1}, \nabla^{\spinor \sphere^{n+1}}\right)$ be the spinor bundle with spinor connection over the $(n+1)$-sphere and let $(E, \nabla^{E,s})$ be the pullback bundle along $f$ with the family of connections defined in \eqref{Eq::FamilyOfConnections}.
Let  
\begin{equation*}
	\Dirac_{E,s}\colon L^2_G(\spinor N\otimes E)\supset\Lip_\cc(\spinor N\otimes E) \to L^2_G(\spinor N\otimes E)
\end{equation*}
be the associated family of twisted Dirac operators.
Again, we suppress the isometry $\sigma^s$, defined in \eqref{FamilyOfLTwoIsometries}, from the notation.

Due to \cref{ConeDiracSAFredholmOp} and \cref{ConeDiracGapTopology}, the minimal closed extension $\cl\Dirac_{E,s}$ is a self-adjoint Fredholm operator for every $s\in [0,1]$ and the path $[0,1]\ni s\mapsto \cl\Dirac_{E,s}$ is continuous in the gap topology.
Therefore, the spectral flow is well-defined and \cref{SpecFlowFormula} yields
\begin{equation*}
	\abs{\specflow\left([0,1]\ni s\mapsto \cl\Dirac_{E,s}\right)} = \abs{\deg(f)} \neq 0.
\end{equation*}
Hence there is an $s_0\in [0,1]$ such that $\cl\Dirac_{E,s_0}$ is not invertible, and we find a non-zero harmonic spinor $u\in \ker(\cl\Dirac_{E,s_0})$.

Let $\curvOp^E_s$ be the curvature endomorphism on $\spinor N\otimes E$ associated with the connection $\nabla^{E,s}$.
Using the fact that $f$ is area non-increasing, a standard estimate for the curvature endomorphism shows that for every $p\in N$ where $f$ is differentiable and every $s\in [0,1]$, it holds (see \cite{LiSuWang}*{Equation (3.9)})
\begin{equation*}
	\innerprod{\curvOp^E_s v}{v} \geq -s(s-1)(n+1)n\cdot \abs{v}^2\geq -\tfrac14 (n+1)n \cdot \abs{v}^2\qquad \text{for all } v\in \left(\spinor N\otimes E\right)_p.
\end{equation*}
The equality case occurs for every $v\in \left(\spinor N\otimes E\right)_p$ if and only if $s=\frac12$ and $d_pf$ is an isometry.
If we are in the equality situation, it follows from \cite{Baer_2024}*{Lemma 2} that\footnote{In \cite{Baer_2024}, the parameter $s$ is shifted by $-\frac12$ and runs over $\left[-\frac12,\frac12\right]$ instead of $[0,1]$.}
\begin{equation}\label{IsometryWithPairOfTangentVectors}
	\left(X\cdot Y\otimes d_pf\left(X^G\right)\cdot d_pf\left(Y^G\right)\right)\cdot v  = v
\end{equation}
for all $\Gamma$-orthonormal vectors $X,Y\in T_pN$ and for all $v\in \left(\spinor N\otimes E\right)_p$.
Recall the notation $X^G = b_G(X)$ introduced in \eqref{NotationMetricMediator}.

From the integrated Schr\"odinger-Lichnerowicz formula for manifolds with cone-like singularities (see \cite{CHSS}*{Theorem 4.12}) and the assumption $\scal_G \geq (n+1)n$, we obtain
\begin{equation*}
	0 = \norm{\cl\Dirac_{E,s_0}u}_{L^2}^2 = \norm{\nabla^{E,s_0} u}_{L^2}^2 + \tfrac14\innerprod{\scal_G u}{u} + \innerprod{\curvOp^E_{s_0} u}{u}_{L^2} \geq 0.
\end{equation*}
Consequently, we get $\scal_G \equiv (n+1)n$, and $f$ is a local isometry.

The rest of the proof is similar to \cite{CHSS}*{Section 4.5}.
However, for the sake of completeness, we repeat the argument. 

Recall that $\metriccl f\colon \metriccl N\to \sphere^{n+1}$ denotes the unique continuous extension of $f$ to the metric completion $\metriccl N$ of $N$ and $\metriccl N \setminus N = \{x_1, \ldots, x_\ell\}$ is the set of cone points of $\metriccl N$. 
Put
\begin{equation*}
	\Sigma := f^{-1}\left(\left\{ \metriccl f(x_1), \ldots, \metriccl f(x_\ell) \right\}\right)\subset N.
\end{equation*}
We showed that
\begin{equation*}
	f\vert_{N\setminus \Sigma}\colon N\setminus \Sigma \to \sphere^{n+1}\setminus \left\{\metriccl f(x_1), \ldots, \metriccl f(x_\ell)\right\}
\end{equation*}
is a local isometry. 
Moreover, as $\metriccl{f}\colon \metriccl{N}\to \sphere^{n+1}$ is proper, its restriction $f\vert_{N\setminus \Sigma}$ is also proper and $\sphere^{n+1}\setminus \left\{\metriccl f(x_1), \ldots, \metriccl f(x_\ell)\right\}$ is simply connected.

To get from local isometry to global isometry, we want to apply \cite{CHS}*{Theorem 2.4}.
In order to do so, it remains to show that $d_pf$ is either orientation-preserving or orientation-reversing for every $p\in N$ where $f$ is differentiable.
The argument relies on \eqref{IsometryWithPairOfTangentVectors} and carries over from \cite{Baer_2024}*{Section 5.7} to our non-compact setting since the computations are local and the argument involves elliptic regularity, which also holds in the non-compact setting (see \cite{GilbargTrudinger}*{Corollary 8.11}). 
It follows that $f$ is a metric isometry. 

Assume that $\Sigma \neq \emptyset$ and let $p\in \Sigma$ such that $f(p) = \metriccl{f}(x_i)$ for some $i = 1, \ldots, \ell$.
By continuity of $\metriccl{f}$, we find $q\in N$ near $x_i\in \metriccl{N}$ such that 
\begin{equation*}
	\dist_N(p,q) > \dist_{\sphere^{n+1}\setminus \left\{\metriccl f(x_1), \ldots, \metriccl f(x_\ell)\right\}}(f(p), f(q)).
\end{equation*}
This is a contradiction and hence $\Sigma = \emptyset$.

The Myers-Steenrod theorem implies that $f$ is a Riemannian isometry since the involved metrics are smooth. 
This concludes the proof.

\bibliographystyle{plain}
\bibliography{references}

\end{document}